\documentclass[reqno,10pt,a4paper]{amsart}
\usepackage[margin=2.5cm]{geometry}
\usepackage{amssymb}
\usepackage{cite}
\usepackage{lmodern}
\usepackage{mathtools}
\usepackage[usenames,dvipsnames]{color}
\usepackage{hyperref}
\hypersetup{
 colorlinks,
 linkcolor={red!50!black},
 citecolor={blue!50!black},
 urlcolor={blue!80!black}
}
\usepackage{tikz}

\AtBeginDocument{%
\def\MR#1{}
}

\newcommand{\bx}{\boldsymbol{x}}
\newcommand{\bmu}{\boldsymbol{\mu}}
\newcommand{\balpha}{\boldsymbol{\alpha}}
\newcommand{\blambda}{\boldsymbol{\lambda}}
\newcommand{\Q}{\mathbb{Q}}
\newcommand{\R}{\mathbb{R}}
\newcommand{\V}{\mathcal{V}}
\newcommand{\Z}{\mathbb{Z}}
\newcommand{\PP}{\mathbb{P}}
\newcommand{\NR}{N_\R}
\newcommand{\abs}[1]{\left|{#1}\right|}
\newcommand{\hV}{\widehat{\V}}
\newcommand{\aff}{\operatorname{aff}}
\newcommand{\conv}{\operatorname{conv}}
\newcommand{\D}{\Delta}
\newcommand{\RR}{R}
\newcommand{\st}{\,\colon}
\newcommand{\orig}{\boldsymbol{0}}
\newcommand{\Hom}{\operatorname{Hom}}
\newcommand{\GL}{\operatorname{GL}}
\newcommand{\intr}{\operatorname{int}}
\newcommand{\lin}{\operatorname{lin}}
\newcommand{\ve}{\operatorname{vert}}
\newcommand{\vol}{\operatorname{vol}}
\newcommand{\Vol}{\operatorname{Vol}}
\newcommand{\supp}{\operatorname{supp}}
\renewcommand{\emptyset}{\varnothing}
\newtheorem{thm}{Theorem}[section]
\newtheorem{prop}[thm]{Proposition}
\newtheorem{cor}[thm]{Corollary}
\newtheorem{lem}[thm]{Lemma}
\theoremstyle{definition}
\newtheorem{defn}[thm]{Definition}

\numberwithin{equation}{section}
\begin{document}
\author[G.\,Balletti]{Gabriele Balletti}
\address{Department of Mathematics\\Stockholm University\\SE-$106$\ $91$\ Stockholm\\Sweden}
\email{balletti@math.su.se}
\author[A.\,M.\,Kasprzyk]{Alexander M.\,Kasprzyk}
\address{School of Mathematical Sciences\\University of Nottingham\\Nottingham NG7 2RD\\UK}
\email{a.m.kasprzyk@nottingham.ac.uk}
\author[B.\,Nill]{Benjamin~Nill}
\address{Fakult\"at f\"ur Mathematik\\Institut f\"ur Algebra und Geometrie\\Otto-von-Guericke-Universit\"at Magdeburg\\Universit\"atsplatz 2\\ 39106 Magdeburg\\Germany}
\email{benjamin.nill@ovgu.de}
\title{On the maximum dual volume of a canonical Fano polytope}
\maketitle
\begin{abstract}
We give an upper bound on the volume $\vol(P^*)$ of a polytope $P^*$ dual to a $d$-dimensional lattice polytope $P$ with exactly one interior lattice point, in each dimension $d$. This bound, expressed in terms of the Sylvester sequence, is sharp, and is achieved by the dual to a particular reflexive simplex. Our result implies a sharp upper bound on the volume of a $d$-dimensional reflexive polytope. Translated into toric geometry, this gives a sharp upper bound on the anti-canonical degree $(-K_X)^d$ of a $d$-dimensional toric Fano variety $X$ with at worst canonical singularities.
\end{abstract}
\section{Introduction}\label{sec:intro}
\subsection{Background and results}
Let $N\cong\Z^d$ be a lattice of rank $d$. A convex polytope $P\subset\NR$, where $\NR:=N\otimes_\Z\R\cong\R^d$, is called a \emph{lattice polytope} if the vertices $\ve(P)$ of $P$ are contained in $N$. Two lattice polytopes $P,Q\subset\NR$ are said to be \emph{unimodular equivalent} if there exists an affine lattice automorphism $\varphi\in\GL_d(\Z)\ltimes\Z^d$ of $N$ such that $\varphi_\R(P)=Q$. Unless stated otherwise, we regard lattice polytopes as being defined only up to unimodular equivalence.

Let $P\subset\NR$ be a lattice polytope of dimension $d$ (that is, $P$ is of maximum dimension in $\NR$) containing exactly one lattice point in its (strict) interior,~i.e.\ $\abs{\intr(P)\cap N}=1$. We can assume that this interior point is the origin $\orig\in N$. For reasons that are explained in~\S\ref{subsec:alg_geom} below, we call $P$ a \emph{canonical Fano polytope}. As a consequence of results by Hensley~\cite[Theorem~3.6]{Hen83} and Lagarias--Ziegler~\cite[Theorem~2]{LZ91}, there are finitely many canonical Fano polytopes (up to unimodular equivalence) in each dimension $d$.

Canonical Fano polytopes in dimensions $d\leq 3$ have been classified~\cite{Kas10}, and we find that $\vol(P)\leq12$. For $d\geq 4$ it is conjectured that the volume of a $d$-dimensional canonical Fano polytope is bounded by
\begin{equation}\label{eq:vol_P_bound}
\vol(P)\leq\frac{1}{d!}2(s_d-1)^2,
\end{equation}
where $s_i$ denotes the $i$-th term of the \emph{Sylvester sequence}:
\[
s_1:=2,\quad s_{i+1}:=s_1\cdots s_i+1\text{ for } i\in\Z_{\ge 1}.
\]
Moreover, the case of equality in~\eqref{eq:vol_P_bound} is expected to be attained only by the canonical Fano simplex
\[
\RR_{(d)}:=S_{(d)}-\sum_{i=1}^d e_i,\qquad\text{ where } S_{(d)} :=\conv\{\orig, 2(s_d-1)e_d, s_{d-1} e_{d-1},\ldots, s_1 e_1\}.
\]
Here $\{e_1,\ldots,e_d\}$ is a basis of $N$. This conjecture is hinted at in~\cite{Reid82,ZPW82,LZ91}, explicitly stated in~\cite[Conjecture~1.7]{Nil07}, and proved by Averkov--Kr\"{u}mpelmann--Nill~\cite{AKN14} for the case when $P$ is a canonical Fano simplex. The conjecture remains open for a general canonical Fano polytope. The currently best upper bound on the volume of a canonical Fano polytope that is not a simplex is established in~\cite[Theorem~2.7]{AKN14} (improving upon a result by Pikhurko~\cite{Pik01}), however this is presumed to be far from sharp:
\[
\vol(P)\leq (s_{d+1}-1)^d.
\]

Instead of bounding $\vol(P)$, it is also natural to consider the volume of the dual polytope $P^*$ (see~\S\ref{notationnow} for the definition of the dual polytope). The main result of this paper is:

\begin{thm}\label{thm:main}
Let $P\subset\NR$ be a $d$-dimensional canonical Fano polytope, where $d\geq 4$. Then
\[
\vol(P^*)\leq\frac{1}{d!}2(s_d-1)^2,
\]
with equality if and only if $P=\RR_{(d)}^*$.
\end{thm}

In three dimensions, the expected bound $\vol(P^*)\leq 12$ is proved in~\cite[Theorem~4.6]{Kas10}. In this case, however, equality is obtained by the duals of two distinct simplices:
\begin{equation}\label{eq:maximal_degree_dim_3}
P_{1,1,1,3}=\conv\{e_1,e_2,e_3,-e_1-e_2-3e_3\}\qquad\text{ and }\qquad P_{1,1,4,6}=\RR_{(3)}^*.
\end{equation}
The analogue of Theorem~\ref{thm:main} is proved in~\cite[Theorem~2.5(b)]{AKN14} for $d$-dimensional canonical Fano simplices.

Probably one of the most studied class of canonical Fano polytopes are the \emph{reflexive polytopes}, consisting of those $P\subset\NR$ such that the dual $P^*$ is also a canonical Fano polytope (for a brief survey see~\cite{KN13}). Note that $\RR_{(d)}$ is a reflexive simplex~\cite{Nil07}. An immediate consequence of Theorem~\ref{thm:main} is a proof of the conjectured inequality~\eqref{eq:vol_P_bound} in the case of reflexive polytopes:

\begin{cor}\label{cor:vol_P_bound_reflexive}
Let $P\subset\NR$ be a $d$-dimensional reflexive polytope, where $d\geq 4$. Then
\[
\vol(P)\leq\frac{1}{d!}2(s_d-1)^2,
\]
with equality if and only if $P=\RR_{(d)}$.
\end{cor}
\noindent
The analogue of Corollary~\ref{cor:vol_P_bound_reflexive} in the case of reflexive simplices is proved in~\cite[Theorem~A]{Nil07}.

\subsection{Toric geometry and Fano varieties}\label{subsec:alg_geom}
Canonical Fano polytopes arise naturally in algebraic geometry. To each $d$-dimensional canonical Fano polytope $P\subset\NR$ we can associate a $d$-dimensional projective toric variety $X_P$ whose fan is given by the cones in $\NR$ spanning the faces of $P$ (here we require that the unique interior point of $P$ is taken to be the origin $\orig$ of $N$). This variety is Fano -- recall that a variety $X$ is Fano if its anti-canonical divisor $-K_X$ is ample -- and has at worst canonical singularities. In fact this construction is reversible, and there exists a one-to-one correspondence between (unimodular equivalence classes of) canonical Fano polytopes and (isomorphism classes of) toric Fano varieties with at worst canonical singularities.
For details on canonical singularities and their importance in algebraic geometry, see~\cite{Rei87}; for details on toric geometry, see~\cite{Dan78}; and for additional background material see the survey~\cite{KN13}.

The classification of Fano varieties is a long-standing open problem. An important advance would be to bound the degree $(-K_X)^d$. In the case when $X$ is non-singular the bound
\begin{equation}\label{eq:KMM}
(-K_X)^d\leq\left(3(2^d - 1)(d + 1)^{(d + 1)(2^d - 1)}\right)^d
\end{equation}
was established by Koll\'{a}r--Miyaoka--Mori~\cite{KMM92}, although this is almost certainly not sharp. Very little is known when $X$ has canonical singularities, however Prokhorov~\cite{Pro05} proved that if $X$ is a three-dimensional Fano with Gorenstein canonical singularities then the degree is bounded by $(-K_X)^3\leq 72$. In this case the maximum degree is obtained by the two weighted projective spaces $\PP(1,1,1,3)$ and $\PP(1,1,4,6)$, and these two toric varieties correspond to the two canonical Fano simplices in~\eqref{eq:maximal_degree_dim_3}. It is tempting to conjecture that, in higher dimensions, the maximum degree is obtained by a toric Fano variety. Recalling that $(-K_{X_P})^d=d!\vol(P^*)$, Theorem~\ref{thm:main} provides a sharp bound on the degree when $X$ is toric:

\begin{cor}\label{cor:bound_canonical_Fano_degree}
Let $X$ be a $d$-dimensional toric Fano variety with at worst canonical singularities, where $d\ge 4$. Then
\begin{equation}\label{eq:canonical_degree_bound}
(-K_X)^d\leq 2 (s_d-1)^2,
\end{equation}
with equality if and only if $X$ is isomorphic to the weighted projective space
\[
\PP\left(1,1,2 (s_d-1) / s_{d-1},\ldots, 2 (s_d-1)/s_1\right).
\]
\end{cor}
\noindent
This extends~\cite[Theorem~A]{Nil07} and~\cite[Theorem~2.11]{AKN14}, where analogous results are stated when $X$ is a Gorenstein fake weighted projective space, and when $X$ is a fake weighted projective space with at worst canonical singularities, respectively. Corollary~\ref{cor:bound_canonical_Fano_degree} also generalises the three-dimensional bound of~\cite[Theorem~4.6]{Kas10}.

Finally, Corollary~\ref{cor:bound_canonical_Fano_degree} also has implications for current attempts to classify non-singular Fano varieties via Mirror Symmetry~\cite{CCGK16}. Here the hope is that a non-singular Fano variety $X$ with $-K_X$ very ample has a $\Q$-Gorenstein deformation to a Gorenstein canonical toric Fano variety $X_P$. Since this deformation would leave the degree unchanged, so the bound of Corollary~\ref{cor:bound_canonical_Fano_degree} would apply to $X$. It is interesting to note that, in this case, the bound~\eqref{eq:canonical_degree_bound} is significantly smaller that the bound~\eqref{eq:KMM} of Koll\'{a}r--Miyaoka--Mori.

\subsection{Overview of the proof}
Our strategy to prove Theorem~\ref{thm:main} is as follows. In~\S\ref{sec:2} we reduce the problem to canonical Fano polytopes satisfying some minimality condition. We observe that such polytopes admit a decomposition into canonical Fano simplices (following~\cite{Kas10}; compare also with the decomposition used in~\cite{KS97}), for which the statement is already known~\cite{AKN14}. In~\S\ref{sec:monotonicity} we use this decomposition, together with the monotonicity of the normalised volume, to prove Theorem~\ref{thm:main} in the majority of cases (Corollary~\ref{cor:almostallcases}). Finally, the remaining cases are proved in~\S\ref{sec:final} using a mixture of integration techniques (developed in~\S\S\ref{sec:slicing}--\ref{sec:integration}) and explicit classification.

\subsection{Notation and terminology}\label{notationnow}
Let $P\subset\NR$ be a lattice polytope of maximum dimension in a rank $d$ lattice $N\cong\Z^d$, and let $M:=\Hom_{\Z}(N,\Z)\cong\Z^d$ be the lattice dual to $N$. The \emph{dual} (or \emph{polar}) polyhedron of $P$ is:
\[
P^*:=\{y\in M_\R\st\langle y,x\rangle\geq-1\text{ for every }x\in P\}.
\]
If $\orig\in P$ then $P^*$ is a convex polytope, although typically $P^*$ has rational vertices, and so is not a lattice polytope.

Let $P$ and $Q$ be two maximum-dimensional polytopes in $(N_P)_\R\cong\R^p$ and $(N_Q)_\R\cong\R^q$, respectively. Suppose that $P$ and $Q$ contain the origin $\orig_{P}\in N_P$ and $\orig_{Q}\in N_Q$ of their respective ambient space. The \emph{free sum} (or \emph{direct sum}) is the maximum-dimensional polytope
\[
P\oplus Q:=\conv((P\times\{\orig_Q\})\cup(\{\orig_P\}\times Q))\subset\R^{p+q}.
\]
The \emph{product} is the polytope
\[
P\times Q:=\{(x_p,x_q)\st x_p\in P,\, x_q\in Q\}\subset\R^{p+q}.
\]
Free sums and products of polytopes are related, via duality, by:
\[
(P\oplus Q )^*=P^*\times Q^*. 
\]

On the affine hull $\aff(P)$ there exists a volume form, called the \emph{relative lattice volume}, that is normalised by setting the volume of a fundamental parallelepiped of $\aff_\Z(P)$ equal to $1$. We denote the relative lattice volume of $P$ by $\vol_N(P)$. The volume $\Vol_N(P):=\dim(P)!\vol_N(P)$ is often called the \emph{normalised lattice volume} of P. If $N'\subseteq N$ is a sublattice of $N$ then, for $S\subseteq\lin(N')$, we have $\vol_{N'}(S)\leq\vol_N(S)$. If in addition we have that $N'\to N$ splits over $\Z$, then $\vol_{N'}(S)=\vol_N(S)$.

\section{Decomposition of minimal polytopes}\label{sec:2}
The case of canonical Fano simplices is already considered in~\cite{AKN14}. Our focus is on the case when $P$ is not a simplex. Notice that if $P\subsetneq Q$ then $Q^*\subsetneq P^*$, and hence $\vol(Q^*)<\vol(P^*)$. It is therefore sufficient to prove Theorem~\ref{thm:main} for ``small'' polytopes $P$; that is, for the \emph{minimal} canonical Fano polytopes:

\begin{defn}[\!\!{\cite[Definition~2.2]{Kas10}}]
A $d$-dimensional canonical Fano polytope $P\subset\NR$ is \emph{minimal} if, for each vertex of $P$, the polytope obtained by removing this vertex is not a $d$-dimensional canonical Fano polytope; that is, if $\conv(P\cap N\setminus\{v\})$ is not a $d$-dimensional canonical Fano polytope, for each $v\in\ve(P)$.
\end{defn}
\noindent
Each canonical Fano polytope $Q$ can, via successive removal of vertices, be reduced to a minimal polytope $P\subset Q$. Of course $P$ need not be uniquely determined. Minimal canonical Fano polytopes admit a decomposition in terms of lower-dimensional minimal canonical Fano simplices:

\begin{prop}[\!\!{\cite[Proposition~3.2]{Kas10}}]\label{prop:Kasp}
Let $P$ be a minimal canonical Fano $d$-polytope that is not a simplex. Then there exists a minimal canonical Fano $k$-simplex $S$ contained in $P$ with $\ve(S)\subset\ve(P)$, for some $1\le k<d$. For any such $S$ there exists a minimal canonical Fano $(d-k+s)$-polytope $P'$ with $\ve(P')\subset\ve(P)$ such that $P=\conv(S\cup P')$, $s=\abs{\ve(S)\cap\ve(P')}$, and $0\leq s < k$.
\end{prop}
\noindent
For brevity we write ``$d$-polytope'' rather than ``polytope of dimension $d$'', and ``$k$-simplex'' rather than ``simplex of dimension $k$''.

\begin{cor}\label{cor:Kasp}
Let $P$ be a minimal canonical Fano $d$-polytope that is not a simplex. Then, for some $2\le t\le d$, there exist minimal canonical Fano simplices $S_1,\ldots,S_t$ such that $P=\conv(S_1\cup\ldots\cup S_t)$, where $\dim(S_i)=d_i\geq 1$ and $\ve(S_i)\subset\ve(P)$, for each $1\leq i\leq t$. Set $r_1:=0$ and, for each $2\leq i\leq t$, set $r_i:=\abs{\ve(S_i)\cap\ve(P^{(i-1)})}$, where $P^{(i-1)}:=\conv(S_1\cup\ldots\cup S_{i-1})$. Then:
\begin{align}
\label{eq:dr}
d_1+\cdots+d_t &= d + r,&\quad\text{ where }r:=r_1 +\cdots + r_t;\\
\label{eq:di}
r_i < d_i &\leq d -t +1,&\quad\text{ for each }1\leq i\leq t;\\
\label{vertices}
\abs{\ve(P)}&=d+t.&
\end{align}
\end{cor}
\noindent
An example of this decomposition is illustrated in Figure~\ref{fig:1}.

\begin{proof}
We apply Proposition~\ref{prop:Kasp} iteratively, at each step choosing $S$ to be of smallest possible dimension. Thus $P$ can be written as $P=\conv(S_1\cup\ldots\cup S_t)$ for some $t\geq 1$, where the $S_i$ are minimal canonical Fano simplices of dimension $d_i\ge 1$ with $\ve(S_i)\subseteq\ve(P)$ having $r_i$ common vertices with $P^{(i-1)}$, such that $d_t\le d_{t-1}\le\cdots\le d_1$. The case $P^{(0)}$ is taken to be the empty set, giving $r_1=0$. At each step, the dimension of $P^{(i)}$ can be obtained from Proposition~\ref{prop:Kasp}: $\dim(P^{(i)})=\dim(P^{(i-1)})+\dim(S_i)-r_i$. Hence $d=\sum_{i=1}^t(d_i-r_i)$, and so~\eqref{eq:dr} holds. Once again using Proposition~\ref{prop:Kasp}, since $\dim(S_i)>r_i$, so $\dim(P^{(i)})\geq\dim(P^{(i-1)})+1$. It follows that $t\leq d$ and so $d_1\leq d-t+1$. Hence our choice of simplices implies~\eqref{eq:di}. Finally, the number of vertices of $P^{(i)}$ is $\abs{\ve(P^{(i-1)})}+\abs{\ve(S_i)}-r_i$. This implies that $\abs{\ve(P)}=\sum_{i=1}^t(d_i+1)-r$, and from~\eqref{eq:dr} we deduce that equation~\eqref{vertices} holds.
\end{proof}
\noindent
Notice that equality~\eqref{vertices}, combined with the bound $t\leq d$, implies that a minimal canonical Fano polytope $P$ satisfies $\abs{\ve(P)}\leq 2 d$ (this is known as \emph{Steinitz's inequality}).

\begin{figure}[tb]
\centering
\begin{tikzpicture}[scale=0.8]
\draw[fill,black!15] (2,-1)--(3.5,0)--(2.5,2.5)--cycle;
\draw[fill,black!35] (2.8,0.4)--(5.5,0.3)--(2.5,2.5)--cycle;
\draw[fill,black!35] (0,0.5)--(2.2,0.4)--(2.5,2.5)--cycle;
\draw[thick] (0,0.5)--(2,-1)--(5.5,0.3);
\draw[dashed] (2,-1)--(3.5,0)--(5.5,0.3);
\draw[thick] (2.5,2.5)--(2,-1);
\draw[thick] (2.5,2.5)--(0,0.5);
\draw[dashed] (2.5,2.5)--(3.5,0)--(0,0.5);
\draw[thick] (2.5,2.5)--(5.5,0.3);
\draw[fill=black] (2.73,1) circle (0.2em) node[below left]{$\orig$};
\draw (1.3,2.2) node{$P$};
\draw (1.4,0.9) node{$S_1$};
\draw (2.5,-0.25) node{$S_2$};
\draw (2.5,2.5) node[above]{$v$};
\end{tikzpicture}
\caption{An example of a three-dimensional minimal canonical Fano polytope $P$ which decomposes into two canonical Fano simplices $S_1$ and $S_2$ sharing a common vertex $v$. In the notation of Corollary~\ref{cor:Kasp}, $d=3$, $t=2$, $d_1=d_2=2$, and $r_2=1$.}\label{fig:1} 
\end{figure}
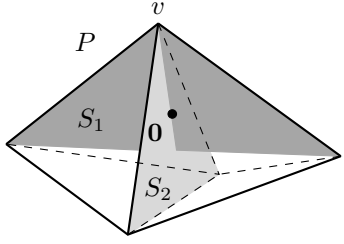

\section{Bounding the volume of $P^*$ via monotonicity of the normalised volume}\label{sec:monotonicity}
As noted above, it is sufficient to prove Theorem~\ref{thm:main} for minimal canonical Fano polytopes that are not simplices. Let $P\subset\NR$ be such a polytope of dimension $d\geq 4$. Fix a decomposition of $P$, and use the notation $t,S_i,d_i,r_i,r$ as defined in Corollary~\ref{cor:Kasp}. In this section we prove Theorem~\ref{thm:main} for the majority of decompositions. The decompositions \emph{not} addressed in this section, and whose proof is the focus of~\S\S\ref{sec:slicing}--\ref{sec:integration} below, are listed in Corollary~\ref{cor:almostallcases}.

\begin{cor}\label{cor:almostallcases}
In order to prove Theorem~\ref{thm:main} it is enough to verify that the inequality
\[
\vol(P^*)\leq\frac{1}{d!}2(s_d-1)^2
\]
holds for all minimal canonical Fano polytopes $P\subset\NR$ of dimension $d\ge 4$ whose decomposition into minimal canonical Fano simplices falls into one of the following five cases:
\begin{enumerate}
\item $t=2$ and $d_1=d_2=d-1$; or
\item $t=2$, $d=4$, $d_1=3$, and $d_2=2$; or
\item $t=2$, $d=5$, $d_1=4$, and $d_2=3$; or
\item $t=3$, $d=4$, and $d_1=d_2=d_3=2$; or
\item $t=3$, $d=5$, and $d_1=d_2=d_3=3$.
\end{enumerate}
\end{cor}
\noindent
In order to prove Corollary~\ref{cor:almostallcases} we use the monotonicity of the normalised volume.
Let $N_i:=\lin_{\R}(S_i)\cap N$ be the sublattice of lattice points in the linear hull of $S_i$ (recall that $\orig\in\intr(S_i)$, so this really is a sublattice), for each $1\leq i\le t$. Define the map
\[
\varphi\colon N_1\oplus\cdots\oplus N_t\to N,\qquad(x_1,\ldots,x_t)\mapsto\sum_{i=1}^{t} x_i .
\]
Notice that $\varphi$ may not be surjective, however since its image has the same rank as $N$, the extension $\varphi_\R$ of $\varphi$ to a map of vector spaces is surjective. Moreover, $\varphi_\R$ gives the following representation of $P$:
\[
P=\varphi_\R(S_1\oplus\cdots\oplus S_t).
\]
Let $M,M_1,\ldots,M_t$ denote the lattices dual to $N,N_1,\ldots,N_t$, respectively. The map $\varphi_\R^*$ dual to $\varphi_\R$ is an injection, and in particular
\[
P^*\cong\varphi^*_\R(P^*)\subset(S_1\oplus\cdots\oplus S_t)^*=S_1^*\times\cdots\times S_t^*,
\]
where $M$ is naturally embedded via $\varphi^*$ into $M_1\oplus\cdots\oplus M_t$. This situation will be studied in more detail in~\S\ref{sec:slicing}. Using the monotonicity of the normalised volume, finding an upper bound for the normalised volume of $S_1^*\times\cdots\times S_t^*$ yields an upper bound for the normalised volume of $P^*$. Specifically, we know that:
\begin{align}\label{eq:geq3}
\begin{split}
\Vol_M(P^*)&\leq\Vol_{M_1\oplus\cdots\oplus M_t}(S_1^*\times\cdots\times S_t^*)\\
&=(d_1+\cdots+d_t)!\vol_{M_1\oplus\cdots\oplus M_t}(S_1^*\times\cdots\times S_t^*)\\
&=(d_1+\cdots+d_t)!\prod_{i=1}^t\vol_{M_i}(S_i^*)\\
&=\frac{(d_1+\cdots+d_t)!}{d_1!\cdots d_t!}\prod_{i=1}^t\Vol_{M_i}(S_i^*).
\end{split}
\end{align}
The normalised volume of $S_i^*$ is bounded from above (see~\cite{Kas10} and~\cite[Theorem~2.5(b)]{AKN14}):
\[
\Vol_{M_i}(S^*_i)\leq B_i,\qquad\text{with}\;
B_i:=\left\{
\begin{array}{ll}
  9 &\text{if}\;d_i=2,\\
  2(s_{d_i}-1)^2 &\text{if}\;d_i\neq 2.
\end{array}
\right.
\]
Hence inequality~\eqref{eq:geq3} becomes:
\[
\Vol_M(P^*)\leq\frac{(d_1+\cdots+d_t)!}{d_1!\cdots d_t!}\,\prod_{i=1}^t B_i.
\]
At this point, Theorem~\ref{thm:main} would follow from
\begin{equation}\label{eq:provet3}
\frac{(d_1+\cdots+d_t)!}{d_1!\cdots d_t!}\,\prod_{i=1}^t B_i< B_d.
\end{equation}
Unfortunately inequality~\eqref{eq:provet3} does not always hold: for example, it fails when $t=2$ and $d_1=d_2=d-1$, for any $d\geq 3$. Nevertheless, this technique is sufficient to prove Theorem~\ref{thm:main} for a large number of cases:

\begin{lem}\label{lem:provet3_cases}
Inequality~\eqref{eq:provet3} -- and therefore Theorem~\ref{thm:main} -- holds whenever:
\begin{enumerate}
\item\label{item:first_case}
$t\geq 3$, with the exception of the following six cases:
\begin{enumerate}
	\item $t=3$, $d=4$, and $d_1=d_2=d_3=2$; or
	\item $t=3$, $d=5$, and $d_1=d_2=d_3=3$; or
	\item $t=3$, $d=4$, $d_1=d_2=2$, and $d_3=1$; or
	\item $t=3$, $d=5$, $d_1=d_2=3$, and $d_3=2$; or
	\item $t=3$, $d=6$, and $d_1=d_2=d_3=4$; or
	\item $t=4$, $d=5$, and $d_1=d_2=d_3=d_4=2$;
\end{enumerate}
\item\label{item:second_case}
$t=2$, with the exceptions of the following three cases:
\begin{enumerate}
	\item $d_1=d_2=d-1$; or
	\item $d=4$, $d_1=3$, and $d_2=2$; or
	\item $d=5$, $d_1=4$, and $d_2=3$.
\end{enumerate}
\end{enumerate}
\end{lem} 
\begin{proof}
We prove~\eqref{item:first_case} and~\eqref{item:second_case} separately, but by the same general technique: first we show that the statement is true for large values of $d$; then we check the finite number of remaining values.
\begin{enumerate}
\item[\eqref{item:first_case}]
Since the quantity
\[
\frac{(d_1 +\cdots + d_t)!}{d_1!\cdots d_t!}\,\prod_{i=1}^t B_i
\]
increases as the $d_i$ increase, by~\eqref{eq:di} it is enough to prove inequality~\eqref{eq:provet3} when $d_i=d-t+1$ for all $i$. That is, it is sufficient to show that
\begin{equation}\label{eq:naive}
\frac{(t(d-t+1))!}{(d-t+1)!^t}\, (B_{d-t+1})^{t} < B_d.
\end{equation}
From $n!\leq 2\cdot 2^2\cdots 2^{n-1} = 2^{n(n-1)/2}$ (which is strict when $n\ge 3$) we obtain:
\[
\frac{(t(d-t+1))!}{(d-t+1)!^t}\le (t(d-t+1))! < 2^{\frac{1}{2}t(d-t+1)(t(d-t+1)-1)}.
\]
Therefore, if the inequality
\begin{equation}\label{eq:toprove}
2^{\frac{1}{2}t(d-t+1)(t(d-t+1)-1)}2^t(B_{d-t+1})^{t}\leq B_d
\end{equation}
holds, so too does inequality~\eqref{eq:naive}.

To prove~\eqref{eq:toprove} we make use of the well-known description due to Aho--Sloane~\cite[Example~2.5]{AS73} and Vardi of the Sylvester sequence in terms of the constant $c\approx 1.2640847353\dots$:
\[
s_n=\left\lfloor c^{2^n}+\frac{1}{2}\right\rfloor.
\]
Notice that $B_d=2(s_d-1)^2>(s_d + 1)^2$ whenever $d\geq 3$. Since $s_d+1>c^{2^d}$, the right-hand side of~\eqref{eq:toprove} is bounded from below:
\[
B_d=2(s_d-1)^2 > (s_d + 1)^2 > c^{2^{d+1}}.
\]
Moreover $B_{d-t+1}/2 < c^{2^{d-t+2}}$. Since $c^3 > 2$, the left-hand side of~\eqref{eq:toprove} is bounded from above:
\[
2^{\frac{1}{2}t(d-t+1)(t(d-t+1)-1)} 2^t\left(\frac{B_{d-t+1}}{2}\right)^{t} <c^{\frac{3}{2}t(d-t+1)(t(d-t+1)-1)} c^{3t} c^{2^{d-t+2}t}.
\]
We shall show that $c^{\frac{3}{2}t(d-t+1)(t(d-t+1)-1)} c^{3t} c^{2^{d-t+2}t}\leq c^{2^{d+1}}$, from which we conclude that inequality~\eqref{eq:toprove} holds. Taking $\log_c$, we have to verify that the inequality
\begin{equation*}
\frac{3}{2}t(d-t+1)(t(d-t+1)-1) + 3t + 2^{d-t+2} t\leq 2^{d+1}
\end{equation*}
is satisfied. Rewrite this inequality as:
\begin{equation}\label{eq:exp}
3t(d-t+1)(t(d-t+1)-1) + 6t\leq 2^{d+2}\left(1-\frac{t}{2^{t-1}}\right).
\end{equation}
Since $t\ge 3$, by setting $t=3$ in the right-most factor it is enough to prove that:
\begin{equation*}
3t(d-t+1)(t(d-t+1)-1) + 6t\leq 2^{d}.
\end{equation*}
Since $t(d-t+1)$ is maximised when $t=(d+1)/2$, and since $6t\leq 6d$, the above inequality is valid when
\[
\frac{3(d+1)}{2}\left(d-\frac{d+1}{2}+1\right)\left(\frac{d+1}{2}\left(d-\frac{d+1}{2}+1\right)-1\right) + 6d\leq 2^{d}.
\]
This holds when $d\geq 13$. Recalling that $d$ bounds the quantities $t,d_1,\ldots,d_t$, we are left with finitely many cases to verify. Inequality~\eqref{eq:provet3} holds in all but six cases, as listed in the statement.

\item[\eqref{item:second_case}]
By the same monotonicity argument used at the beginning of the previous case, we choose $d_1$ and $d_2$ as great as possible, i.e.~we fix $d_1=d-1$ and $d_2=d-2$ (we noted above that inequality~\eqref{eq:provet3} is not satisfied when $d_1=d_2=d-1$). Inequality~\eqref{eq:provet3} becomes
\begin{equation}\label{eq:eq2}
\frac{(2d-3)!}{(d-2)! (d-1)!}\, B_{d-2} B_{d-1} < B_d.
\end{equation}
Proceeding as above, we reduce the problem to proving the inequality
\[
3(2d^2-7d+8)+2^{d-1}+2^d\le 2^{d+1}.
\]
This holds when $d\geq 10$. Removing the assumptions $d_1=d-1$ and $d_2=d-2$ on $d_1$ and $d_2$, the finitely many cases for $4\leq d\leq 9$ can be directly verified against inequality~\eqref{eq:provet3}. We find the exceptional cases listed in the statement of the Lemma.\qedhere
\end{enumerate}
\end{proof}

\begin{proof}[Proof of Corollary~\ref{cor:almostallcases}]
By Lemma~\ref{lem:provet3_cases} we need to show that proving Theorem~\ref{thm:main} for all deompositions listed in the statement of Corollary~\ref{cor:almostallcases} also proves it in the four cases:
\begin{enumerate}
\item $t=3$, $d=4$, $d_1=d_2=2$, and $d_3=1$; or
\item $t=3$, $d=5$, $d_1=d_2=3$, and $d_3=2$; or
\item $t=3$, $d=6$, and $d_1=d_2=d_3=4$; or
\item $t=4$, $d=5$, and $d_1=d_2=d_3=d_4=2$.
\end{enumerate}
In each cases we have that either $t=3$ or $t=4$. By Corollary~\ref{cor:Kasp} we can express $P$ as $P = P'\cup S_t$, where $P' = S_1\cup\ldots\cup S_{t-1}$ is a minimal polytope of dimension $d'$ decomposed into $t'=t-1$ minimal simplices. Note that in all four cases $d'=d-1$. We now proceed exactly as in the first part of this section. Let $N':=\lin_\R(P')\cap N$ be the sublattice of $N$ of lattice points in the linear hull of $P'$. We define the map $\varphi'\colon N'\oplus N_{t}\to N$ by $(x_1,x_2)\mapsto x_1 + x_2$, whose extension $\varphi_\R$ to a map of vector spaces is surjective and gives the following representation of $P$:
\[
P=\varphi_\R(P'\oplus S_{t}).
\]
Let $M'$ denote the lattice dual to $N'$. The map $(\varphi')_\R^*$ dual to $(\varphi')_\R$ is an injection, and in particular
\[
P^*\cong(\varphi')_\R^*(P^*)\subset(P'\oplus S_{t})^*=(P')^*\times S_{t}^*.
\]
As in~\eqref{eq:geq3}, by the monotonicity of the normalised volume,
\begin{align}\label{eq:geq3cor}
\begin{split}
\Vol_M(P^*)&\leq\frac{(d'+d_{t})!}{d'!d_{t}!}\Vol_{M'}((P')^*)\Vol_{M_{t}}(S_{t}^*).
\end{split}
\end{align}
By our assumption and Lemma~\ref{lem:provet3_cases}, Theorem~\ref{thm:main} holds for $t'=2$ and for $t'=3$, $d'=4$, $d_1=d_2=d_3=2$. Hence, in all four cases, Theorem~\ref{thm:main} holds for $P'$; i.e.~$\Vol_{M'}((P')^*) < B_{d-1}$. Since $\Vol_{M_{t}}(S_{t}^*)\leq B_{d_{t}}$ and $d'=d-1$,
\[
\Vol_M(P^*) <\frac{(d-1+d_{t})!}{(d-1)!d_{t}!} B_{d-1} B_{d_{t}}.
\]
Hence it is enough to prove that
\[
\frac{(d-1+d_{t})!}{(d-1)!d_{t}!} B_{d-1} B_{d_{t}} < B_d.
\]
This inequality can be directly checked in all four cases.
\end{proof}

\section{Slicing minimal polytopes}\label{sec:slicing}
We now develop the foundations for a finer technique that we use in~\S\ref{sec:final} to help prove the remaining cases of Theorem~\ref{thm:main}. In particular, we shall explain how minimal polytopes can be described as a particular union of slices which are products of slices of simplices (see Figure~\ref{fig:slice}). Using this construction, in~\S\ref{sec:integration} we give a better estimate of the dual volume via integration.

\subsection{Embedding the dual polytope}\label{subsec:slicing1}
As above, we are in the setup of Corollary~\ref{cor:Kasp}: $P\subset\NR$ is a $d$-dimensional  minimal canonical Fano polytope decomposed into  minimal canonical Fano simplices $S_1,\ldots,S_t$, for some $t\geq 2$. We define
\[
\V:=\!\!\bigcup_{1\leq i_1 < i_2\leq t}\!\!\ve(S_{i_1})\cap\ve(S_{i_2})
\]
to be the set of those vertices of $P$ which occur multiple times amongst the vertices of the $S_i$, and define $\V_i:=\mathcal{V}\cap\ve(S_{i})$. For example, in Figure~\ref{fig:1} we have $\V=\V_1=\V_2=\{v\}$. 

\begin{figure}[tb]
\centering
\begin{minipage}[t]{0.3\textwidth}
\begin{tikzpicture}[scale=0.6]
\draw[dashed] (-1,0)--(2,1)--(1.7,2.8);
\draw[dashed] (2,1)--(5,1);
\draw[fill,black!20] (0.25,1.3)--(2.36,1.3)--(3.94,1.86)--(1.86,1.86);
\draw[thick,black] (1.7,2.8)--(2.8,2.8);
\draw[thick,black] (-1,0)--(2,0);
\draw[thick,black] (2,0)--(5,1)--(2.8,2.8)--cycle;
\draw[thick,black] (1.7,2.8)--(-1,0);
\draw[fill=black] (2.15,1.6) circle (0.2em) node[above]{$\orig$};
\end{tikzpicture}
\end{minipage}
\begin{minipage}[t]{0.25\textwidth}
\begin{tikzpicture}[scale=0.6]
\draw[very thick,black!50] (1.07,1.58)--(3.17,1.58);
\draw[thick,black!30] (3.5,0.5)--(2.4,4.1)--(0.5,0.5)--cycle;
\draw[very thick,black!50] (3.5,0.5)--(0.5,0.5);
\draw[fill=black] (2.15,1.6) circle (0.2em) node[above]{$\orig$};
\draw (4.1,0.8) node[below]{{\tiny $F^*_1$}};
\draw (3,1.6) node[right]{{\tiny $H_{1,\orig}$}};
\end{tikzpicture}
\end{minipage}
\begin{minipage}[t]{0.25\textwidth}
\begin{tikzpicture}[scale=0.6]
\draw[very thick,black!50] (1.30,1.3)--(2.88,1.86);
\draw[thick,black!30] (3.5,1)--(2.25,2.8)--(0.5,0)--cycle;
\draw[very thick,black!50] (3.5,1)--(0.5,0);
\draw[fill=black] (2.15,1.6) circle (0.2em) node[above]{$\orig$};
\draw (1.2,0.7) node[right]{{\tiny $F^*_2$}};
\draw (0.6,1.6) node[right]{{\tiny $H_{2,\orig}$}};
\end{tikzpicture}
\end{minipage}
\caption{The dual $P^*$ of the polytope $P$ from Figure~\ref{fig:1}, together with the dual triangles $(S'_1)^*\subset(M'_1)_\R$ and $(S'_2)^*\subset(M'_2)_\R$. In the first picture, the grey slice is $H_{1,\orig}\times H_{2,\orig}$. We refer to~\S\ref{subsec:slicing2} for the precise definitions.\label{fig:slice}}
\end{figure}
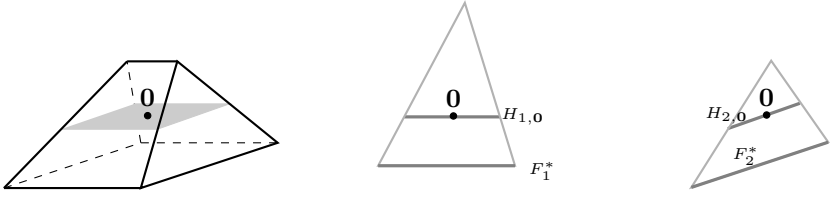

It will be convenient to coarsen the lattice $N$. We note that coarsening the ambient lattice $N$ to a lattice $N'$ is an assumption we can make. Indeed, if $P_M^*$ and $P_{M'}^*$ denote the duals of $P$ with respect to the lattices $M=N^*$ and $M'=(N')^*$, respectively, then the volume of $P_{M'}^*$ is equal to the volume of $P_M^*$ multiplied by the index of $M'$ as a subgroup of $M$ (which is a positive integer). 

Let $N'_i$ denote some sublattice of $N_i=\lin_{\R}(S_i)\cap N$ of same rank $d_i$ with $\V_i\subset N'_i$ (a specific choice of $N'_i$ will be given in~\S\ref{subsec:slicing2}). Notice that $S_i$ may no longer be a \emph{lattice} simplex with respect to $N'_i$. Therefore, in order to avoid any confusion, we denote by $S'_i\subseteq (N'_i)_\R=(N_i)_\R$ the \emph{rational} simplex with vertices $\ve(S_i)$ with respect to the lattice $N'_i$. Now, by possibly coarsening the lattice $N$ we may suppose that $N$ is the image of the lattice $N'_1\oplus\cdots\oplus N'_t$ via the map
\begin{equation}\label{phimap}
\begin{array}{r@{\ }c@{\ }l}
\varphi\colon N'_1\oplus\cdots\oplus N'_t &\to&N\\
(x_1,\ldots,x_t) &\mapsto&{\displaystyle\sum_{i=1}^{t}x_i}.
\end{array}
\end{equation}
Hence we can assume that this map is surjective. Notice that the polytope $P$ may no longer be a lattice polytope with respect to this ambient lattice. We extend the map $\varphi$ to the map of real vector spaces $\varphi_\R\colon(N'_1)_\R\oplus\cdots\oplus(N'_t)_\R\to\NR$. As in the previous section we can describe $P$ as
\[
P =\varphi_\R (S'_1\oplus\cdots\oplus S'_t).
\]
By definition $\varphi$ is a surjective map, so we have the exact sequence
\[
0\to\ker\varphi\hookrightarrow N'_1\oplus\cdots\oplus N'_t\twoheadrightarrow N\to 0,
\]
which splits over $\Z$. From~\eqref{eq:dr} we have that $N'_1\oplus\cdots\oplus N'_t$ splits in parts of rank $d$ and $r$. As a consequence, the dual sequence
\[
0\to M\hookrightarrow M'_1\oplus\cdots\oplus M'_t\twoheadrightarrow (\ker\varphi)^*\to 0
\]
is exact and splits too. Here we used the notation $M'_1,\ldots,M'_t$ for the dual lattices of $N'_1,\ldots,N'_t$, respectively. Let $(\ker\varphi)^\perp$ denote the elements of $M'_1\oplus\cdots\oplus M'_t$ vanishing on $\ker\varphi$. By the exactness of the dual sequence, $\varphi^*(M)=(\ker\varphi)^\perp$; that is, the lattices $M$ and $(\ker\varphi)^\perp$ are isomorphic via $\varphi^*$. In particular, $(\ker\varphi)^\perp = (M'_1\oplus\cdots\oplus M'_t)\cap (\ker\varphi)^\perp_\R$ is a direct summand of $M'_1\oplus\cdots\oplus M'_t$ of rank $d$.

By tensoring by $\R$ to extend the maps to the ambient real vector spaces, it follows that the following polytopes are isomorphic as rational polytopes with respect to their respective lattices:
\begin{align}\label{eq:P*}
\begin{split}
P^* &\cong\varphi_\R ^* (P^*)\\
&=(S'_1\oplus\cdots\oplus S'_t)^*\cap (\ker\varphi)_\R^\perp\\
&=((S'_1)^*\times\cdots\times (S'_t)^*)\cap (\ker\varphi)_\R^\perp.	
\end{split}
\end{align}

We now describe a set of generators of $(\ker\varphi)_\R$. For this, let us identify $N'_i$ with the corresponding direct summand in $N'_1\oplus\cdots\oplus N'_t$. In this way, we can identify $v\in\ve(S'_i)$ with $e_{i,v}\in N'_1\oplus\cdots\oplus N'_t$, i.e., $(e_{i,v})_i = v\in N'_i$ and $(e_{i,v})_j =\orig_{N'_j}$ for $j\not=i$. Recall that $\dim_\R (\ker\varphi)_\R = r$. Let $1\leq i_1 < i_2\leq t$, and $v\in\V_{i_1}\cap\V_{i_2}$. We denote by $w_{v,i_1,i_2}$ the element $e_{i_2,v}-e_{i_1,v}\in N'_1\oplus\cdots\oplus N'_t$.

\begin{lem}\label{lem:generators}
With notation as above, $\ker\varphi_\R$ is generated by the set
\[
\Omega :=\{w_{v,i_1,i_2}\in N'_1\oplus\cdots\oplus N'_t\st 1\leq i_1 < i_2\leq t,\, v\in\V_{i_1}\cap\V_{i_2}\}.
\]
\end{lem}
\begin{proof}
We prove that the subset
\[
\Omega':=\left\{w_{v,i_1,i_2}\in\Omega\st i_1=\max\{i\st v\in\V_i, i<i_2\}\right\}
\]
of $\Omega$ is a basis of $\ker\varphi_\R$. Since for $2\le i\leq t$ we have $\abs{\{w_{v,i_1,i}\in\Omega'\st i_2=i\}}=r_i$, this implies that $\abs{\Omega'}=\sum_{i=2}^t r_i = r$. Hence it is enough to prove that the elements of $\Omega'$ are linearly independent.

Denote the elements of $\Omega'$ by $\bx_1,\ldots,\bx_r$, where $\bx_j=((\bx_j)_1,\ldots,(\bx_j)_t)\in N'_1\oplus\cdots\oplus N'_t$. Assume there exists a nontrivial relation $\mu_1\bx_1 +\ldots +\mu_r\bx_r =\orig$ with $\bmu := (\mu_1,\ldots,\mu_r)\in\R^r\setminus\{(0,\ldots,0)\}$. Let us define
\[
\supp(\bmu) :=\{j\in\{1,\ldots, r\}\st\mu_j\neq 0\}.
\]
Let $i\in\{1,\ldots, t\}$ be the largest such integer such that there exists an integer $j\in\supp(\bmu)$, an index $1\le i_1 < i$, and a vertex $v\in\V_{i_1}\cap\V_i$, with $w_{v,i_1,i}=\bx_j$. By definition of $i$ and $\Omega'$, all elements in $\{(\bx_j)_i\st j\in\supp(\bmu), (\bx_j)_i\not=\orig_{N'_i}\}\neq\emptyset$ are pairwise different vertices in $\V_i\cap\ve(P^{(i-1)})$. Hence
\[
\sum_{j\in\supp(\bmu)}\mu_j\, (\bx_j)_i =\orig_{N'_i}
\]
implies a nontrivial relation of a non-empty subset of the vertices in $\V_i\cap\ve(P^{(i-1)})$. However, as $S_i$ contains the origin in its interior, any proper subset of the set of vertices of $S_i$ is linearly independent, so $\V_i\cap\ve(P^{(i-1)}) =\ve(S_i)$. Hence, $r_i=d_i+1$, a contradiction to~\eqref{eq:di}.
\end{proof}
\noindent
We now apply Lemma~\ref{lem:generators} to~\eqref{eq:P*}:
\begin{align}
\label{eq:slicing1}
\begin{split}
P^*&\cong\varphi^*_\R (P^*)\\
&=((S'_1)^*\times\cdots\times (S'_t)^*)\cap (\ker\varphi)_\R^\perp\\
&=\{(y_1,\ldots , y_t)\in (S'_1)^*\times\cdots\times (S'_t)^*\st\langle (y_1,\ldots,y_t) ,\omega\rangle = 0\text{ for each }\omega\in (\ker\varphi)_\R\}\\
&=\{(y_1,\ldots , y_t)\in (S'_1)^*\times\cdots\times (S'_t)^*\st\langle y_{i_1} , e_{i_1,v}\rangle =\langle y_{i_2} , e_{i_2,v}\rangle\text{ for each } w_{v,i_1,i_2}\in\Omega\}\\
&=\{(y_1,\ldots , y_t)\in (S'_1)^*\times\cdots\times (S'_t)^*\st\langle y_{i_1} , e_{i_1,v}\rangle =\langle y_{i_2} , e_{i_2,v}\rangle\text{ for each } v\in\V_{i_1}\cap\V_{i_2}\}.\\
\end{split}
\end{align}

\subsection{The integration map}\label{subsec:slicing2}
From here onwards we will assume that the decomposition of $P$ into the simplices $S_i$ is \emph{irredundant}, i.e.~$\V_i\subsetneq\ve(S_i)$ for $i=1,\ldots, t$. Under this assumption, we describe a specific choice for $N'_i$. For this, we choose a vertex $v_i\in\ve(S_i)\setminus\V_i$, and set 
\[
\hV_i:=\ve(S_i)\setminus\{v_i\}.
\]
We have $\V_i\subset\hV_i$. We define $N'_i$ to be the lattice spanned by $\hV_i$, that is,
\[
N'_i:=\langle v\in\hV_i\rangle_\Z.
\]
By construction, the $d_i$ vertices in $\hV_i$ form a lattice basis
\[
\{e_{i,v}\}_{v\in\hV_i}
\]
of $N'_i$ (as a sublattice of $N'_1\oplus\cdots\oplus N'_t$). 
Note that the vertex $v_i$ need not be a lattice point in $N'_i$. We again assume that $N$ is given as the image of $\varphi$, see~\eqref{phimap}, and we will refer to $S_i$ as $S'_i$ when referring to it with respect to the lattice $N'_i$. This choice of lattice will allow us to prove Lemma~\ref{lem:psisurj} which simplifies the considerations in~\S\ref{sec:integration}. In particular, it will yield a convenient explicit description of $(S'_i)^*$ (see Lemma~\ref{lem:S*}).

Set $q :=\abs{\V}$ and $q_i :=\abs{\V_i}$ for $i=1,\ldots, t$. We define $\Psi$ to be the map
\begin{align*}
\Psi: (\ker\varphi)^\perp &\to\bigoplus_{v\in\V}\Z\cong\Z^q\\
 (y_1,\ldots,y_t) &\mapsto (\langle y_{i_v} , e_{i_v,v}\rangle)_{v\in\V},
\end{align*}
where, for each $v$, $i_v$ is any index such that $v\in\V_{i_v}$. Since $\langle y_{i_1} , e_{i_1,v}\rangle=\langle y_{i_2} , e_{i_2,v}\rangle$ whenever $v\in\V_{i_1}\cap\V_{i_2}$, $\Psi$ is a well defined map. In an analogous fashion to the definition of $\Psi$, for each $i\in\{1,\ldots,t\}$ we define the map
\begin{align*}
\Psi_i: M'_i &\to\bigoplus_{v\in\V_i}\Z\cong\Z^{q_i}\\
 y &\mapsto (\langle y , e_{i,v}\rangle)_{v\in\V_i}.
\end{align*}

\begin{lem}\label{lem:psisurj}
The maps $\Psi,\Psi_1,\ldots,\Psi_t$ are surjective.
\end{lem}
\begin{proof}
Let $\{\epsilon_{i,v}\}_{v\in\V_i}$ be the standard basis of $\bigoplus_{v\in\V_i}\Z$, and $\{e_{i,v}^*\}_{v\in\hV_i}$ the lattice basis of $M'_i$ dual to the lattice basis $\{e_{i,v}\}_{v\in\hV_i}$ of $N'_i$. The maps $\Psi_i$ are surjective, since each element $e_{i,v}^*$ is mapped into $\epsilon_{i,v}$, for $v\in\V_i$.

We now prove that $\Psi$ is surjective. Since the codomains of the maps $\Psi_i$ span the codomain of $\Psi$, it is enough to check that for each $i\in\{1,\ldots,t\}$ and for each $v\in\V_i$, there exists an element $(y_1,\ldots,y_t)\in (\ker\varphi)^\perp\subset M'_1\oplus\cdots\oplus M'_t$, such that $y_i=e_{i,v}^*$. This is true, since it suffices to choose $(y_1,\ldots,y_t)$ as
\[\sum_{j\text{ such that }v\in V_j}\!\!e_{j,v}^*\in M'_1\oplus\cdots\oplus M'_t.\qedhere
\]
\end{proof}
\noindent
As a consequence of Lemma~\ref{lem:psisurj} the extensions of $\Psi,\Psi_1,\ldots,\Psi_t$ to the real vector space maps
\[
\Psi_\R,(\Psi_1)_\R,\ldots,(\Psi_t)_\R
\]
are linear surjective maps. We define natural projections
\[
p_i:\bigoplus_{v\in\V}\R\to\bigoplus_{v\in\V_i}\R
\]
as the identity over $\bigoplus_{v\in\V_i}\R$ and the zero map over $\bigoplus_{v\in\V\setminus\V_i}\R$. 

Let $\mathcal{D}$ be the set of parameters
\[
\mathcal{D}:=\Psi_\R(\varphi^*_\R (P^*))\subset\bigoplus_{v\in\V}\R.
\]
Given a point $\blambda=(\lambda_v)_{v\in\V}\in\mathcal{D}$, define the fibre
\[
H_{i,\blambda}	:= (\Psi_i)_\R^{-1}(p_i(\blambda))\cap (S'_i)^* =\{y\in (S'_i)^*\st\langle y , v\rangle =\lambda_v\text{ for all }v\in\V_i\}\subset (M'_i)_\R.
\]
Denote by $F^*_i$ the $(d_i-q_i)$-dimensional face of $(S'_i)^*$ given by
\begin{equation}
\label{eq:Fi}
F^*_i:=H_{i,(-1,\ldots,-1)}.
\end{equation}
From~\eqref{eq:slicing1} we obtain the desired decomposition of $P^*$:
\begin{align}
\label{eq:slicing2}
\begin{split}
P^* &\cong\bigsqcup_{(\lambda_v)_{v\in\V}\in\mathcal{D}}\{(y_1,\ldots , y_t)\in (S'_1)^*\times\cdots\times (S'_t)^*\st\langle y_i , e_{i,v}\rangle =\lambda_v\text{ for all } v\in\V_i,\, i = 1,\ldots, t\}\\
&=\bigsqcup_{\blambda\in\mathcal{D}} H_{1,\blambda}\times\cdots\times H_{t,\blambda}\\
\end{split}
\end{align}
In other words, $P^*$ is sliced into a disjoint union of sections (see Figure~\ref{fig:slice}).

\section{Bounding the volume of $P^*$ via integration}\label{sec:integration}
In this section we apply~\eqref{eq:slicing2} to obtain a finer bound on the volume of $P^*$ in the case when $P$ decomposes into just two simplices. From here onwards we assume we are in the setup of Corollary~\ref{cor:Kasp} with $t=2$,~i.e.\ $P$ decomposes in two minimal canonical simplices $S_1$ and $S_2$ of dimensions $d_1$ and $d_2$ respectively. As $P$ is not a simplex, clearly this decomposition is irredundant, so the results of~\S\ref{subsec:slicing2} apply. We will continue to use the notation introduced in~\S\ref{sec:slicing}, and in particular the choice of $N'_i,N,S'_i$ in~\S\ref{subsec:slicing2}. Note that $q=r_2=r = |\V| = |\V_1| = |\V_2|$ is the number of common vertices of $S_1$ and $S_2$. 

Equality~\eqref{eq:slicing2} and Lemma~\ref{lem:psisurj} allow us to calculate the volume $\vol_M(P^*)$ by integrating the sections over the possible values of $\blambda$. In particular:
\begin{equation}\label{eq:int1}
\vol_M(P^*) =\int_{\blambda\in\mathcal{D}}\vol_{M'_1}(H_{1,\blambda})\vol_{M'_2}(H_{2,\blambda})\,d\blambda.
\end{equation}
Before attempting to bound such a value, we present an alternative description of $\mathcal{D}$. For $i=1,2$, we define $\mathcal{D}_i$ as
\[
\mathcal{D}_i:=(\Psi_i)_\R ((S'_i)^*),
\]
and we note that (since the maps $p_i$ defined in the previous section correspond to the identity maps here),
\begin{equation}
\label{eq:lambda}
\mathcal{D}=\mathcal{D}_1\cap\mathcal{D}_2.
\end{equation}

Recall that a lattice basis $\{e_{i,v}\}$ for $N'_i$ is given by all the elements of $\hV_i=\ve(S_i)\setminus\{v_i\}$. Denote by $(\beta_{i,v})_{v\in\ve(S_i)}$ the barycentric coordinates of the origin in the simplex $S_i$,~i.e.\ $\sum_{v\in\ve(S_i)}\beta_{i,v} v =\orig$, where $\sum_{v\in\ve(S_i)}\beta_{i,v} 
= 1$. Note that $\beta_{i,v} > 0$ for any $v\in\ve(S_i)$. Hence, we can express $v_i$ as 
\[v_i= -\sum_{v\in\hV_i}\frac{\beta_{i,v}}{\beta_{i,v_i}}e_{i,v}.\]

Let us denote by $\{\epsilon_{i,v}\}_{v\in\V_i}$ the standard basis of $\bigoplus_{v\in\V_i}\Z$. 
Lemma~\ref{lem:S*} below gives an explicit description for $(S'_i)^*$ and $\mathcal{D}_i$ in terms of our chosen lattice bases. We omit the straightforward proof. 

\begin{lem}\label{lem:S*}
With notation as above, for $i=1,2$
\[
(S'_i)^* =\conv\left(\left\{-\sum_{v\in\V} e^*_{i,v}\right\}\cup\left\{\left(\frac{1}{\beta_{i,w}}-1\right)e^*_{i,w}-\sum_{v\in\V\setminus\{w\}} e^*_{i,v}\right\}_{w\in\hV_i}\right),
\]
\[
\mathcal{D}_i =\conv\left(\left\{-\sum_{v\in\V}\epsilon^*_{i,v}\right\}\cup\left\{\left(\frac{1}{\beta_{i,w}}-1\right)\epsilon^*_{i,w}-\sum_{v\in\V\setminus\{w\}}\epsilon^*_{i,v}\right\}_{w\in\V}\right).
\]
\end{lem}

By using the inequality $f_1 f_2\leq\frac{f_1^2+f_2^2}{2}$ we can bound~\eqref{eq:int1} via
\begin{align}
\begin{split}
\vol_{M}(P^*)	&\leq\int_{\blambda\in\mathcal{D}}\frac{\vol_{M'_1}(H_{1,\blambda})^2 +\vol_{M'_2}(H_{2,\blambda})^2}{2}\,d\blambda\\
&=\frac{1}{2}\int_{\blambda\in\mathcal{D}}\vol_{M'_1}(H_{1,\blambda})^2\,d\blambda +\frac{1}{2}\int_{\blambda\in\mathcal{D}}\vol_{M'_2}(H_{2,\blambda})^2\,d\blambda\\
&\leq\frac{1}{2}\int_{\blambda\in\mathcal{D}_1}\vol_{M'_1}(H_{1,\blambda})^2\,d\blambda +\frac{1}{2}\int_{\blambda\in\mathcal{D}_2}\vol_{M'_2}(H_{2,\blambda})^2\,d\blambda,\\
\end{split}
\label{nextiq}
\end{align}
where the final inequality follows from~\eqref{eq:lambda}. It is convenient to perform a change of variables for $i=1,2$, via the maps 
\[
\balpha=(\alpha_v)_{v\in\V}\xmapsto{f_i} (\frac{1}{\beta_{i,v}}\alpha_v - 1)_{v\in\V}.
\]
By Lemma~\ref{lem:S*}, the integration domain $\mathcal{D}_i$ becomes the unimodular $q$-dimensional simplex $\D_{(q)}$; that is, the convex hull of the origin and the standard basis of $\Z^q$. Hence~\eqref{nextiq} can be rewritten as:
\begin{equation}\label{eq:int2}
\vol_{M}(P^*)\leq\frac{1}{2}\prod_{v\in\V}\frac{1}{\beta_{1,v}}\int_{\balpha\in\D_{(q)}}\!\!\vol_{M'_1}(H_{1,f_1(\balpha)})^2\,d\balpha +\frac{1}{2}\prod_{v\in\V}\frac{1}{\beta_{2,v}}\int_{\balpha\in\D_{(q)}}\!\!\vol_{M'_2}(H_{2,f_2(\balpha)})^2\,d\balpha.\\
\end{equation}

\begin{lem}[\!\!{\cite[Lemma~3.5~III]{Ave12}}]
With notation as above, for $i=1,2$,
\[
\vol_{M'_i}(H_{i,f_i(\balpha)})=\vol_{M'_i}(F^*_i)\left( 1 -\sum_{v\in\V_i}\alpha_v\right) ^{d_i-q},
\]
where $F_i$ is the $(d_i-q)$-dimensional face of $(S'_i)^*$ defined in~\eqref{eq:Fi}.
\end{lem}
\noindent
Inequality~\eqref{eq:int2} can now be rewritten as
\begin{align}\label{eq:int3}
\begin{split}
\vol_{M}(P^*)\leq &\frac{1}{2}\prod_{v\in\V}\frac{1}{\beta_{1,v}}\vol_{M'_1}(F^*_1)^t\int_{\balpha\in\D_{(q)}}\left( 1 -\sum_{v\in\V}\alpha_v\right) ^{2(d_1-q)}\,d\balpha\\ +&
\frac{1}{2}\prod_{v\in\V}\frac{1}{\beta_{2,v}}\vol_{M'_2}(F^*_2)^t\int_{\balpha\in\D_{(q)}}\left( 1 -\sum_{v\in\V}\alpha_v\right) ^{2(d_2-q)}\,d\balpha.\\
\end{split}
\end{align}

The following Lemma derives from a special case of a well-known representation of the beta function (see, for example,~\cite[Representation~4.3-2]{Car77}).

\begin{lem}\label{lem:beta-representation}
\[
\int_{\balpha\in\D_{(a)}}( 1 -\alpha_1 -\ldots -\alpha_a )^b\,d\balpha =\frac{b!}{(a+b)!}.
\]
\end{lem}
\noindent
Applying Lemma~\ref{lem:beta-representation} to~\eqref{eq:int3} yields:
\begin{equation}\label{eq:int4}
\vol_{M}(P^*)\leq\frac{1}{2}\prod_{v\in\V}\frac{1}{\beta_{1,v}}\vol_{M'_1}(F^*_1)^2\frac{(2(d_1-q))!}{(q+2(d_1-q))!} +\frac{1}{2}\prod_{v\in\V}\frac{1}{\beta_{2,v}}\vol_{M'_2}(F^*_2)^2\frac{(2(d_2-q))!}{(q+(2(d_2-q))!}.\\
\end{equation}

The volume of $F^*_i$ is computed in Lemma~\ref{lem:vol_F_dual} below. Its proof is omitted, since it is a straightforward consequence of the description of $(S'_i)^*$ given in Lemma~\ref{lem:S*}.

\begin{lem}\label{lem:vol_F_dual}
With notation as above, for $i=1,2$,
\[
\vol_{M'_i}(F_i^*)=\frac{1}{(d_i-q)!}\prod_{v\in\hV_i\setminus\V}\frac{1}{\beta_{i,v}}.
\]
\end{lem}
\noindent
Finally, applying Lemma~\ref{lem:vol_F_dual} to~\eqref{eq:int4} gives the following bound for $\vol_{M}(P^*)$:
\begin{align}\label{eq:int5}
\begin{split}
\vol_{M}(P^*)\leq\frac{1}{2}&\frac{(2(d_1-q))!}{(q+2(d_1-q))!((d_1-q)!)^2}\prod_{v\in\V}\frac{1}{\beta_{1,v}}\prod_{v\in\hV_1\setminus\V }\frac{1}{\beta_{1,v}^2}\\
&+\frac{1}{2}\frac{(2(d_2-q))!}{(q+2(d_2-q))!((d_2-q)!)^2}\prod_{v\in\V}\frac{1}{\beta_{2,v}}\prod_{v\in\hV_2\setminus\V }\frac{1}{\beta_{2,v}^2}
\end{split}
\end{align}

\section{Final cases}\label{sec:final}
In this final section we address the remaining cases of Corollary~\ref{cor:almostallcases}. That is, we prove that the decompositions
\begin{enumerate}
\item $t=2$, $d_1=d_2=d-1$, for $d\geq 4$
\item $t=2$, $d_1=d-1$, $d_2=d-2$, $d\in\{4,5\}$
\item $t=3$, $d_1=d_2=d_3=d-2$, $d\in\{4,5\}$
\end{enumerate}
satisfy Theorem~\ref{thm:main}.
\subsection{The case $t=2$, $d_1=d_2=d-1$}\label{sec:t2_d}
By~\eqref{eq:dr} we have $q=d-2$. Hence, inequality~\eqref{eq:int5} can be rewritten as
\begin{equation}\label{eq:lastcase}
\vol_{M}(P^*)\leq\frac{1}{d!}\left(\prod_{v\in\V}\frac{1}{\beta_{1,v}}\prod_{v\in\hV_1\setminus\V }\frac{1}{\beta_{1,v}^2} +\prod_{v\in\V}\frac{1}{\beta_{2,v}}\prod_{v\in\hV_2\setminus\V }\frac{1}{\beta_{2,v}^2}\right)
\end{equation}
We focus on the product
\[
\prod_{v\in\V}\frac{1}{\beta_{i,v}}\prod_{v\in\hV_i\setminus\V }\frac{1}{\beta_{i,v}^2}
\]
for each $i=1,2$. Note that in~\S\ref{subsec:slicing2} we made a choice to exclude one of the vertices (called $v_i$) of $\ve(S_i)\setminus\V$ from appearing in $\hV_i$. As there are two such vertices (say, $\ve(S_i)\setminus\V=\{v_i, u_i\}$), 
we can exclude the one whose corresponding barycentric coordinate is smaller; that is, $\beta_{i,{v_i}}\le\beta_{i,{u_i}}$. This yields
\begin{equation}\label{eq:last}
\prod_{v\in\V}\frac{1}{\beta_{i,v}}\prod_{v\in\hV_i\setminus\V }\frac{1}{\beta_{i,v}^2} =\left(\prod_{v\in\V}\frac{1}{\beta_{i,v}}\right)\frac{1}{\beta_{i,{u_i}}^2}\le\left(\prod_{v\in\V}\frac{1}{\beta_{i,v}}\right)\frac{1}{\beta_{i,{u_i}}}\frac{1}{\beta_{i,{v_i}}} =\frac{1}{\beta_{i,0},\ldots,\beta_{i,d-1}},
\end{equation}
where $\{\beta_{i,v}\st v\in\ve(S_i)\} =\{\beta_{i,j}\st j=0,\ldots, d-1\}$. Notice that equality in~\eqref{eq:last} is attained if and only if $\beta_{i,u_i} =\beta_{i,v_i}$.

For each $i=1,2$, let us sort the barycentric coordinates such that $\beta_{i,0}\ge\beta_{i,1}\ge\cdots\ge\beta_{i,d-1}$.

\begin{lem}[\!\!{\cite[Lemma~4.2(d)]{AKN14}}]\label{lem:bary_iff}
With notation as above,
\[
\frac{1}{\beta_{i,0}\cdots\beta_{i,d-1}}\leq (s_d-1)^2
\]
with equality if and only if
\begin{equation}\label{eq:bary_iff}
\left(\beta_{i,0},\ldots,\beta_{i,d-1}\right) =\left(\frac{1}{s_1},\ldots,\frac{1}{s_{d-1}},\frac{1}{s_d-1}\right).
\end{equation}
\end{lem}
\noindent
Applying Lemma~\ref{lem:bary_iff} and~\eqref{eq:last} to~\eqref{eq:lastcase} we obtain
\[
\vol_{M}(P^*) <\frac{2(s_d-1)^2}{d!}.
\]
This inequality is strict, since the condition that $\beta_{i,u_i} =\beta_{i,v_i}$ from~\eqref{eq:last} and the condition~\eqref{eq:bary_iff} from Lemma~\ref{lem:bary_iff} cannot hold simultaneously.

\subsection{The cases $t=2$, $d_1=d-1$, $d_2=d-2$, $d\in\{4,5\}$}\label{sec:t2_45}
The barycentric coordinates of the canonical Fano simplices up to and including dimension four are classified in~\cite{Kas13}. Hence we can verify that, in this situation, the right hand side of~\eqref{eq:int5} is always strictly less than $2(s_d-1)^2/d!$.

\subsection{The cases $t=3$, $d_1=d_2=d_3=d-2$, $d\in\{4,5\}$}\label{sec:t3_45}
To prove the inequality in these final cases we explicitly construct every minimal polytope $P$ of dimension four or five that admits a decomposition into three minimal simplices of dimensions two or three, respectively. Moreover, we insist that the vertices of $P$ generate the ambient lattice $N$. Indeed, if $P$ is a minimal polytope then $P$ restricted to the lattice generated by the vertices of $P$ is also minimal, and the volume of the dual polytope will have increased. Under this setting we note that $P$ is uniquely determined by:
\begin{enumerate}
\item
the barycentric coordinates of the simplices $S_1,S_2,S_3$ in the decomposition; and
\item
the choice of $d-3$ vertices in common with $S_2$ and $S_1$, together with the choice of $d-3$ vertices in common with $S_3$ and $S_1\cup S_2$.
\end{enumerate}
This follows from the following general construction. The \emph{(reduced) weights} of a canonical Fano simplex $S$ of dimension $n$ are the positive integers $(k\beta_0,\ldots, k\beta_n)$ given by the  barycentric coordinates $(\beta_0,\ldots,\beta_n)$ of the origin (with respect to the vertices of $S$), where $k$ is the smallest positive integer such that the $k\beta_i$ are all integral. In particular, the weights of a canonical Fano simplex are coprime. Moreover, since the vertices of a canonical Fano simplex are primitive lattice points, the weights are also \emph{well-formed}; that is, any $n$ of them are also coprime.

For the construction we use the fact that any minimal polytope $P$ has a decomposition into $t$ minimal simplices. We proceed invariantly, since we do not know the embedding of these simplices into the lattice $N$. Let $\underline{\lambda}^{(n)}=(\lambda_0,\ldots,\lambda_n)$ denote the (reduced, well-formed) weights of a minimal canonical Fano simplex of dimension $n$. Fix weights $\underline{\lambda}^{(d_1)},\ldots,\underline{\lambda}^{(d_t)}$. For each pair $(i,j)$ with $1\leq i < j\leq t$ we pick a (possibly empty) subset $V_{ij}\subset\{0,\ldots,d_i\}\times\{0,\ldots,d_j\}$ such that $V_{ij}:\pi_1(V_{ij})\to\pi_2(V_{ij})$ is a bijection (here $\pi_k$ denotes the projection on the $k$-th factor). Let $\iota_j:\Z^{d_j+1}\to\bigoplus_{i=1}^t\Z^{d_i+1} $, $1\leq j\leq t$, be the natural inclusion on the $j$-th factor. Define:
\begin{align*}
W&:=\left\langle\iota_i(\underline{\lambda}^{(d_i)})\mid 1\leq i\leq d \right\rangle,\\
V&:=\left\langle\iota_i(e_{\pi_1(v)}) -\iota_j(e_{\pi_2(v)})\mid v\in V_{ij}, 1\leq i < j\leq t\right\rangle.
\end{align*}
Applying $-\otimes\R$ ensures torsion-freeness of the quotient $\left(\bigoplus_{i=1}^t\Z^{d_i+1}\right)/(W + V)$, therefore we get the exact sequence
\[
0\to (W + V)\otimes\R\to\left(\bigoplus_{i=1}^t\Z^{d_i+1}\right)\otimes\R\xrightarrow{\varphi_\R} N\otimes\R\to 0,
\]
where $N$ is the lattice obtained as the quotient $\left(\bigoplus_{i=1}^t\Z^{d_i+1}\right)/K$, where $K$ is the direct summand defined by $\left(\bigoplus_{i=1}^t\Z^{d_i+1}\right)\cap\left((W + V)\otimes\R\right)$. We now define
\[
Q:=\varphi_\R\left( \bigoplus_{i=1}^t\iota_i\left(\conv\{e_0,\ldots,e_{d_i}\}\right)\right)\subset N\otimes\R
\]
which by construction is a polytope whose vertices generate its ambient lattice $N$. $Q$ in general may not be a minimal polytope, however, if $P$ is a minimal lattice polytope of dimension $d$ whose vertices generate its ambient lattice then there exists a choice of integers $t,d_1,\ldots,d_t$, weights $\underline{\lambda}^{(d_1)},\ldots,\underline{\lambda}^{(d_t)}$ of minimal Fano simplices $S_1,\ldots,S_t$ of dimensions $d_1,\ldots,d_t$, and subsets $V_{ij}$ (for $1\leq i < j\leq t$) such that the polytope $Q$ constructed above is equal to $P$. The fact that we can recover $P$ from the construction of $Q$ is a consequence of Lemma~\ref{lem:generators}, while existence of the parameters $t,d_1,\ldots,d_t$ and the weights follows from Corollary~\ref{cor:Kasp}. 

We now specialise this construction to the case $t=3$, $d_1=d_2=d_3=d-2$, for $d\in\{4,5\}$. The weights of the minimal canonical Fano simplices of dimension two and three have been classified in~\cite[Figure~1 and Proposition~4.3]{Kas10}. There are two possible weights in dimension two: $(1,1,1)$ and $(1,1,2)$. In dimension three there are $13$ possible weights\footnote{\cite[Proposition~4.3]{Kas10} incorrectly lists $(2,2,3,5)$ as the weight of a minimal canonical Fano simplex, however any such simplex will contain a canonical Fano sub-simplex with weights $(1,1,1,3)$.}, recorded in Table~\ref{tab:dim3}. Since the choices for the common vertices (encoded in the sets $V_{ij}$, $1\leq i < j\leq 3$) are finite, so all the minimal canonical Fano polytopes $P$ admitting such a decomposition and whose vertices generate the ambient lattice $N$ can be classified.

\begin{table}[tb]
\centering
\begin{tabular}{ccccc}
$( 1, 1, 1, 1 )$&$( 1, 1, 1, 2 )$&$( 1, 1, 1, 3 )$&$( 1, 1, 2, 2 )$&$( 1, 1, 2, 3 )$\\
$( 1, 1, 2, 4 )$&$( 1, 1, 3, 4 )$&$( 1, 1, 3, 5 )$&$( 1, 1, 4, 6 )$&$( 1, 2, 3, 5 )$\\
$( 1, 3, 4, 5 )$&$( 2, 3, 5, 7 )$&$( 3, 4, 5, 7 )$&&
\end{tabular}\vspace{0.5em}
\caption{The weights of the minimal canonical Fano simplices in dimension three.}
\label{tab:dim3}
\end{table}

We use the computer algebra system \textsc{Magma}~\cite{BCP97} to derive the classification. Source code and output can be downloaded from the Graded Ring Database~\cite{GRDb}. In the first case ($d_1=d_2=d_3=2$), there are exactly four such four-dimensional polytopes, and in each case the inequality of Theorem~\ref{thm:main} holds. In order to solve the second case ($d_1=d_2=d_3=3$), we first build all possible four-dimensional minimal polytopes $P'$ whose vertices generate the ambient lattice, and admitting a decomposition into two three-dimensional minimal canonical Fano simplices $S_1$ and $S_2$. We then verify that any five-dimensional polytope $P$ decomposing as $S_1$, $S_2$, and $S_3$ satisfies inequality~\eqref{eq:geq3cor} for each choice of three-dimensional minimal canonical Fano simplex $S_3$; that is, we verify that
\[
\Vol(P^*)\leq\frac{7!}{4!\,3!}\Vol(P'^*)\cdot 2(s_3-1)^2 < 2(s_5-1)^2
\]
holds in each case. There are $147$ minimal four-dimensional polytopes with a decomposition into two three-dimensional minimal canonical Fano simplices and whose vertices generate the lattice $N$, and in each case the inequality holds. This completes the proof of Theorem~\ref{thm:main}.

\subsection*{Acknowledgements}
GB is supported by the Stiftelsen GS Magnusons Fund and by a Jubileumsfond grant from the Knut and Alice Wallenbergs Foundation. In addition, both GB and BN are partially supported by Vetenskapsr{\aa}det grant~NT:2014-3991. BN is an affiliated researcher of Stockholm University; he would like to thank the Fields Institute for the financial support to participate in the thematic program ``Combinatorial Algebraic Geometry''. AK is supported by EPSRC grant~EP/N022513/1.
\bibliographystyle{amsplain}
\providecommand{\bysame}{\leavevmode\hbox to3em{\hrulefill}\thinspace}
\providecommand{\MR}{\relax\ifhmode\unskip\space\fi MR }
\providecommand{\MRhref}[2]{%
  \href{http://www.ams.org/mathscinet-getitem?mr=#1}{#2}
}
\providecommand{\href}[2]{#2}

\end{document}